\documentclass[11pt]{article}

\usepackage[margin=1in]{geometry}
\usepackage{amsmath,amssymb,amsthm,mathtools}
\usepackage{enumitem}
\usepackage{hyperref}
\usepackage{microtype}
\usepackage{authblk}
\hypersetup{colorlinks=true,linkcolor=blue,citecolor=blue,urlcolor=blue}

\newtheorem{theorem}{Theorem}[section]
\newtheorem{lemma}[theorem]{Lemma}
\newtheorem{proposition}[theorem]{Proposition}
\newtheorem{corollary}[theorem]{Corollary}
\newtheorem{definition}[theorem]{Definition}
\newtheorem{remark}[theorem]{Remark}
\newtheorem{problem}[theorem]{Problem}
\newtheorem{conjecture}{Conjecture}

\DeclareMathOperator{\odd}{odd}

\title{Matchings and Near-Optimal 2-Factor Packings\\
in Percolated Vertex-Transitive Graphs}

\author[1]{Mengyu Cao\thanks{E-mail: \texttt{myucao@ruc.edu.cn}. Supported by the National Natural Science Foundation of China (12301431) and Beijing Natural Science Foundation (1262010).}}
\author[2]{Mei Lu\thanks{E-mail: \texttt{lumei@tsinghua.edu.cn}. Supported by the National Natural Science Foundation of China (Grant 12571372) and Beijing Natural Science Foundation (Grant 1262010).}}
\author[2]{Xiamiao Zhao\thanks{Corresponding author. E-mail: \texttt{zxm23@mails.tsinghua.edu.cn}}}

\affil[1]{\small Institute for Mathematical Sciences, Renmin University of China, Beijing 100086, China}
\affil[2]{\small Department of Mathematical Sciences, Tsinghua University, Beijing 100084, China}

\date{}

\begin{document}
\maketitle

\begin{abstract}
Let $G$ be a connected simple vertex-transitive graph on $n$ vertices
with degree $d$, and let $G_p$ be the random spanning subgraph obtained
by retaining each edge of $G$ independently with probability $p$.  Put
$q:=1-p$.  Motivated by a conjecture of Bedert, Dragani\'c, M\"uyesser,
and Pavez-Sign\'e on Hamilton cycles in percolated Cayley graphs, we
establish the corresponding matching and $2$-factor statements uniformly
over the larger class of all connected vertex-transitive host graphs.
For every $A>0$, if $q^d\le n^{-(5A+250)},$ then, with probability at least $1-n^{-A}$, the graph $G_p$ has a perfect
matching when $n$ is even and is factor-critical when $n$ is odd.
Separately, if $0<\epsilon<1$ and
$
\epsilon^2pd\ge64(A+6)\log(2n),
$
then, with probability at least $1-n^{-A}$, the graph $G_p$ contains at
least
\[
\left\lfloor\frac{(1-\epsilon)pd}{2}\right\rfloor
\]
pairwise edge-disjoint spanning $2$-factors.  Moreover, if
$pd/\log n\to\infty$, then
\[
\nu_2(G_p)=(1+o(1))\frac{pd}{2}
\]
with high probability, which is asymptotically optimal, where $\nu_2(G)$ is the maximum number of pairwise edge-disjoint spanning
2-factors in $G$.  Thus
logarithmic-order percolation already forces these two factor-theoretic
consequences of Hamiltonicity beyond the Cayley setting.
\end{abstract}



\section{Introduction}

A graph $G$ is \emph{vertex-transitive} if, for every $u,v\in V(G)$, there
is an automorphism $\varphi$ of $G$ such that $\varphi(u)=v$.  This class
contains Cayley graphs, Kneser graphs, and many sparse regular graphs.  The
following classical conjecture asks for a spanning path in every connected
vertex-transitive graph.

\begin{conjecture}[Lov\'asz~\cite{LovaszProblem}]\label{conj:lovasz}
Every connected vertex-transitive graph has a Hamilton path.
\end{conjecture}

A stronger version, often attributed to Thomassen, asserts that every
sufficiently large connected vertex-transitive graph has a Hamilton cycle.
Christofides, Hladk\`y, and M\'ath\'e~\cite{ChristofidesEtAl} proved this for
graphs of order $n$ and degree at least $\eta n$, for every fixed $\eta>0$.
More recently, Buci\'c, Christoph, Pokrovskiy, and
Steiner~\cite{BucicEtAl} proved that every connected vertex-transitive graph
of order $n$ contains a cycle of length at least $n^{2/3-o(1)}$.

An older special case concerns Cayley graphs.  Let $\Gamma$ be a finite
group and  $S\subseteq\Gamma\setminus\{1\}$ with $S=S^{-1}$.  The
Cayley graph $\operatorname{Cay}(\Gamma,S)$ has vertex set $\Gamma$ and
edges $\{g,gs\}$ for $g\in\Gamma$ and $s\in S$.  It is connected exactly
when $S$ generates $\Gamma$.

\begin{conjecture}[Rapaport-Strasser~\cite{RapaportStrasser}]
Every connected Cayley graph on a finite group with at least three elements
has a Hamilton cycle.
\end{conjecture}

Chen and Quimpo~\cite{ChenQuimpo} proved the conjecture for finite abelian
groups.  Christofides, Hladk\`y, and M\'ath\'e~\cite{ChristofidesEtAl} proved
it for all sufficiently large connected Cayley graphs of degree at least
$\eta n$.  Bedert, Dragani\'c, M\"uyesser, and
Pavez-Sign\'e~\cite{BedertEtAl} recently lowered the degree condition to
$d>n^{1-c}$ for an absolute constant $c>0$.

We study a random version of this problem.  For a graph $G$ and
$p\in[0,1]$, let $G_p$ be the random spanning subgraph obtained by retaining
each edge of $G$ independently with probability $p$.  Put $q:=1-p$.  If $G$ is
$d$-regular, then a fixed vertex is isolated in $G_p$ with probability
$q^d$.  Thus $q^d$ is the exact local deletion parameter, and
$pd\asymp\log n$ is the corresponding scale when $p$ is small.

Bedert, Dragani\'c, M\"uyesser, and Pavez-Sign\'e~\cite{BedertEtAl} proposed
the following random analogue.

\begin{conjecture}\label{conj:random-hamilton}
There is an absolute constant $C$ such that, for every connected Cayley
graph $G$ of order $n$ and degree $d$, the condition
\[
        p\ge C\frac{\log n}{d}
\]
implies that $G_p$ has a Hamilton cycle with high probability.
\end{conjecture}

Brada\v{c} and Janzer~\cite{BradacJanzer} proved this conjecture under the
stronger condition
$pd\ge ((n/d)\log n)^C$ for a sufficiently large absolute constant $C$.
The full conjecture remains open.

A Hamilton cycle contains two simpler spanning factors.  It contains a
matching of size $\lfloor n/2\rfloor$.  When $n$ is odd, deletion of any one
vertex from the cycle leaves the graph with a perfect matching.  It is also a
spanning $2$-factor.  These two consequences retain different parts of the
Hamilton-cycle problem.  Matching is controlled by parity barriers, while a
$2$-factor prescribes degree two at every vertex but does not require its
cycles to merge into one component.  Bedert, Dragani\'c, M\"uyesser, and
Pavez-Sign\'e explicitly observed that even the perfect-matching consequence
of Conjecture~\ref{conj:random-hamilton} is of independent interest.

There are two closely related points of comparison.  In the deterministic
matching-preclusion problem, Li, He, and Zhang~\cite{LiHeZhang} proved that
every connected $d$-regular vertex-transitive graph of even order has
matching-preclusion number $d$.  They also classified the exceptional graphs
into six explicit classes for which an optimal preclusion set need not be
the star of a vertex.  Thus
the minimum deterministic obstruction is completely understood, but a
random theorem must additionally count the near-minimum obstructions.
For arbitrary regular bipartite host graphs, Glebov, Luria, and
Simkin~\cite{GlebovLuriaSimkin} proved that, for a $k$-regular bipartite
graph with two parts of order $s$, the perfect-matching and
no-isolated-vertex hitting times coincide asymptotically when
$k=\omega(s/\log^{1/3}s)$.  They also constructed infinitely many regular
bipartite hosts with
$k=\Omega(s/(\log s\log\log s))$ for which isolated vertices disappear
around $p=\log s/k$, while no perfect matching exists with high probability
even at $p=2\log s/k$.  Their examples show that
regularity alone does not support the conclusion below; the transitive
structure of the host is essential.

We settle both factor-theoretic consequences, in a stronger form and for
all connected vertex-transitive graphs.  A graph of even order is
\emph{parity-perfect} if it has a perfect matching.  A graph of odd order is
\emph{parity-perfect} if it is factor-critical, meaning that deleting any vertex
leaves the graph with a perfect matching.  Our first result is uniform in an
auxiliary ambient parameter $N$, a form needed for induction through
vertex-transitive quotients.

\begin{theorem}\label{thm:main-poly}
Let $A>0$, $p\in[0,1]$ and $q=1-p$.  Let $G$ be a connected simple vertex-transitive graph of order
$n\ge2$ and degree $d$.  Let $N\ge n$.  If
\[
        q^d\le N^{-(5A+250)},
\]
then
\[
        \mathbb P(G_p\text{ is parity-perfect})\ge 1-N^{-A}.
\]
\end{theorem}

The numerical constant $250$ is not optimized.  The point of the displayed
form is that the exponent is explicit and linear in the requested error
exponent.  Taking $N=n$ gives the form stated in the abstract.  More
precisely, let $C_*(A)$ be the infimum of the exponents $C$ for which the
conclusion of Theorem~\ref{thm:main-poly} holds uniformly with
$q^d\le N^{-C}$.  Proposition~\ref{prop:cycle-exponent-obstruction} below
shows that odd cycles force $C_*(A)\ge2A+2$.  Consequently, the presently
proved exponent window is
\[
        2A+2\le C_*(A)\le5A+250.
\]
For even-order host graphs alone, the corresponding cycle obstruction gives
the lower bound $C\ge A+2$.

For a graph $H$, its \emph{$2$-factor packing number} $\nu_2(H)$ is the
maximum number of pairwise edge-disjoint spanning $2$-factors in $H$.  The
second result gives an asymptotically optimal packing, rather than only one
$2$-factor.

\begin{theorem}\label{thm:packing}
Let $G$ be a connected simple vertex-transitive graph of order $n\ge3$ and
degree $d$.  Let $A>0$,  $0<\epsilon<1$, and $p\in[0,1]$.  Put
$r:=\left\lfloor(1-\epsilon)pd/2\right\rfloor$.  If
\begin{equation}\label{eq:packing-condition}
        \epsilon^2pd\ge64(A+6)\log(2n),
\end{equation}
then
\[
        \mathbb P\bigl(\nu_2(G_p)\ge r\bigr)\ge1-n^{-A}.
\]
\end{theorem}

The coefficient $1/2$ of $pd$ in the lower bound for $\nu_2(G_p)$ is best
possible.  Every spanning $2$-factor has exactly $n$ edges, whereas $G_p$
has about $pnd/2$ edges.  Theorem~\ref{thm:packing} therefore gives the
correct asymptotic packing number whenever
$pd/\log n\to\infty$.  For comparison, Knox, K\"uhn, and
Osthus~\cite{KnoxKuhnOsthus} proved that the binomial random graph admits an
approximate Hamilton decomposition in the same asymptotic range
$pn/\log n\to\infty$.  Our factors need not be connected, but our conclusion
is uniform over all connected vertex-transitive hosts.

The proofs use two different reductions.  For parity-perfectness, minimal
Tutte obstructions are encoded by a bounded collection of near-minimum cuts
and a near-bipartite colouring; small cuts are then handled by passing to a
block system and inducting on its vertex-transitive quotient.  Tracking the
losses in this obstruction count and quotient recursion produces the
exponent $5A+250$.  For the packing theorem, a random orientation turns the
problem into finding a regular factor in the bipartite double cover.  Ore's
factor criterion and Karger's cut count give the required kernel, which is
then decomposed by Petersen's $2$-factorization theorem.

Together, the two theorems give the following quantitative consequence of
Conjecture~\ref{conj:random-hamilton}.

\begin{corollary}\label{cor:simultaneous}
Let $A>0$ and $p\in[0,1]$. Let $G$ be a connected simple vertex-transitive graph of
order $n\ge3$ and degree $d$.  If
\[
        pd\ge256(A+7)\log(2n),
\]
then, with probability at least $1-n^{-A}$, the graph $G_p$ is
parity-perfect and contains at least $\lfloor pd/4\rfloor$ pairwise
edge-disjoint spanning $2$-factors.
\end{corollary}

\begin{proof}
Apply Theorem~\ref{thm:main-poly} with error exponent $A+1$ and $N=n$,
and apply Theorem~\ref{thm:packing} with error exponent $A+1$ and
$\epsilon=1/2$.  The displayed assumption gives
\[
        \frac{pd}{4}\ge64(A+7)\log(2n)
\]
and
\[
        q^d\le e^{-pd}
        \le(2n)^{-256(A+7)}
        \le n^{-(5A+255)}.
\]
Thus both hypotheses hold.  The sum of the two failure probabilities is at
most $2n^{-(A+1)}\le n^{-A}$.
\end{proof}

\begin{corollary}\label{cor:packing-asymptotic}
Let $(G_n)$ be a sequence of connected vertex-transitive graphs, where
$G_n$ has order $n$ and degree $d_n$, and let $p_n\in[0,1]$.  If
\[
        \frac{p_nd_n}{\log n}\longrightarrow\infty,
\]
then,  as $n\to\infty$,
\[
        \nu_2((G_n)_{p_n})=(1+o(1))\frac{p_nd_n}{2}
\]
with probability tending to $1$.
\end{corollary}

For high-degree hosts, the logarithmic order in the simultaneous corollary
cannot be lowered: when $G=K_n$, the random graph $G_p$ has isolated
vertices with high probability below the usual $\log n/n$ scale and
therefore has neither a spanning $2$-factor nor the relevant matching
property; see~\cite{JansonLuczakRucinski}.  Bounded-degree hosts behave
differently, as the cycle examples in
Proposition~\ref{prop:cycle-exponent-obstruction} show.  The two proofs also
require different structural descriptions of a failure, which is why we
state and prove the results separately rather than deducing one from the
other.

Section~\ref{sec:prelim} collects the common notation and deterministic
tools.  Section~\ref{sec:two-factor} proves Theorem~\ref{thm:packing} and
its consequences.  Section~\ref{sec:parity-obstructions} develops the
parity-obstruction theory and proves Theorem~\ref{thm:main-poly}.  The
final section records extensions, limitations, and the remaining
Hamiltonicity problem.

\section{Preliminaries}\label{sec:prelim}

All graphs are finite and undirected.  Host graphs are simple, but contractions
may produce loopless multigraphs.  For $X\subseteq V(G)$, write
\[
        \partial_G X:=\{xy\in E(G):x\in X,\ y\notin X\}.
\]
For disjoint sets $X,Y\subseteq V(G)$, let $e_G(X,Y)$ be the number of
edges between $X$ and $Y$, and let $e_G(X)$ be the number of edges induced
by $X$.  For a graph $H$, let $\odd(H)$ be the number of odd components of
$H$.  A \emph{$2$-factor} is a spanning $2$-regular subgraph, or
equivalently a vertex-disjoint union of cycles covering every vertex.

This section is organized into three groups of tools.  We first collect the
probabilistic estimate and the matching criteria underlying parity-perfectness.
We then record the structural results on vertex-transitive graphs and small
cuts.  Finally, we give the regular-factor and counting tools used in the two
probabilistic arguments.

\subsection{Probabilistic and parity-perfect criteria}

We use the following standard Chernoff bound.

\begin{lemma}[Janson, \L uczak, and
Ruci\'nski~\cite{JansonLuczakRucinski}]\label{lem:chernoff}
Let $Z=\sum_{i=1}^m Z_i$, where the $Z_i$ are independent Bernoulli random
variables, and put $\mu:=\mathbb E Z$.  For every $0<\eta<1$,
\[
        \mathbb P\bigl(Z<(1-\eta)\mu\bigr)
        \le \exp\left(-\frac{\eta^2\mu}{2}\right).
\]
\end{lemma}

We use Tutte's theorem in the following form.

\begin{theorem}[Tutte~\cite{TutteMatching}]\label{thm:tutte}
A graph $H$ of even order has a perfect matching if and only if
\[
        \odd(H-U)\le |U|
\]
for every $U\subseteq V(H)$.
\end{theorem}

The analogous criterion for factor-criticality follows from
Theorem~\ref{thm:tutte}.

\begin{lemma}\label{lem:factor-critical-criterion}
A graph $H$ of odd order is factor-critical if and only if
\begin{equation}\label{eq:factor-critical-criterion}
        \odd(H-U)\le |U|-1
\end{equation}
for every nonempty set $U\subseteq V(H)$.
\end{lemma}

\begin{proof}
Suppose first that $H$ is factor-critical.  Fix a nonempty set $U\subseteq V(H)$ and
choose $x\in U$.  A perfect matching of $H-x$ must match every odd component
of $H-U$ to a distinct vertex of $U\setminus\{x\}$.  Hence
\eqref{eq:factor-critical-criterion} holds.

Conversely, fix $x\in V(H)$.  For every $T\subseteq V(H)\setminus\{x\}$,
apply \eqref{eq:factor-critical-criterion} to $U=T\cup\{x\}$.  This gives
\[
        \odd((H-x)-T)=\odd(H-U)\le |T|.
\]
By Theorem~\ref{thm:tutte}, the graph $H-x$ has a perfect matching.  Since
$x$ was arbitrary, $H$ is factor-critical.
\end{proof}

Let $|V(H)|\equiv\tau\pmod 2$, where
$\tau\in\{0,1\}$.  Put
\[
        \mathcal U_0(H):=2^{V(H)},
        \qquad
        \mathcal U_1(H):=2^{V(H)}\setminus\{\emptyset\},
\]
and define the \emph{parity deficiency} by
\begin{equation}\label{eq:parity-deficiency}
        \Delta_\tau(H)
        :=\max_{U\in\mathcal U_\tau(H)}
          \bigl(\odd(H-U)-|U|+\tau\bigr).
\end{equation}
A set $U\in\mathcal U_\tau(H)$ attaining the maximum in
\eqref{eq:parity-deficiency} is called a \emph{parity barrier} of $H$.
Thus
\[
\odd(H-U)-|U|+\tau=\Delta_\tau(H)
\]
for every parity barrier $U$.

Each term in \eqref{eq:parity-deficiency} is even, because
\[
\odd(H-U)\equiv |V(H)|-|U|
\equiv\tau+|U|\pmod 2.
\]
If $\tau=0$ and $H$ has a perfect matching, then every component of
$H$ has even order, so the choice $U=\emptyset$ gives value zero.
If $\tau=1$ and $H$ is factor-critical, then for every $x\in V(H)$
the graph $H-x$ has a perfect matching, so the choice $U=\{x\}$ also
gives value zero.  By Theorem~\ref{thm:tutte} and
Lemma~\ref{lem:factor-critical-criterion}, $H$ is parity-perfect if
and only if $\Delta_\tau(H)=0$.

\subsection{Vertex-transitive structure and small cuts}

The next results provide the structural tools used below.  The first gives
the exact edge-connectivity of a vertex-transitive graph.  The second counts
cuts that are close to minimum.  The third supplies many edge-disjoint
spanning trees.  We then quote the precise invariant-factor theorem of
van den Heuvel and Jackson and give a self-contained proof of the small-cut
form used later.

\begin{theorem}[Mader~\cite{Mader} and
Watkins~\cite{Watkins}]\label{thm:mader}
If $G$ is a finite connected $d$-regular vertex-transitive graph, then its edge-connectivity is $d$.  Equivalently,
\[
        |\partial_G X|\ge d
\]
for every nonempty proper set $X\subsetneq V(G)$.
\end{theorem}

\begin{theorem}[Karger~\cite{Karger}]\label{thm:karger}
Let $G$ be an $n$-vertex graph with edge-connectivity $\lambda>0$.  For
every real $\alpha\ge1$, the number of unordered cuts of size at most
$\alpha\lambda$ is less than $n^{2\alpha}$.  Consequently, the number of
nonempty proper sets $X\subsetneq V(G)$ satisfying
$|\partial_GX|\le\alpha\lambda$ is less than $2n^{2\alpha}$.
\end{theorem}

\begin{theorem}[Nash-Williams~\cite{NashWilliams} and
Tutte~\cite{TutteTrees}]\label{thm:treepacking}
If a multigraph $H$ is $\lambda$-edge-connected, then $H$ contains at least $\lfloor \lambda/2\rfloor$ edge-disjoint spanning trees.
\end{theorem}

The following small-cut lemma is closely related to the invariant-factor
theorem of van den Heuvel and Jackson~\cite{vandenHeuvelJackson}.
Since the precise form needed here is slightly different, we give a
self-contained proof.
\begin{lemma}\label{thm:block}
Let $G$ be a connected simple $d$-regular vertex-transitive graph, where
$d\ge4$.  Suppose there is a set $X\subseteq V(G)$ such that
\[
        2\le |X|\le |V(G)|-2
        \quad\text{and}\quad
        |\partial_G X|<2d-2.
\]
Then $\operatorname{Aut}(G)$ has a nontrivial block of imprimitivity $B$
such that
\[
        d\le |B|\le 2d-3,
        \qquad
        |\partial_G B|=|B|,
        \qquad
        G[B]\text{ is connected and $(d-1)$-regular}.
\]
The translates of $B$ form an $\operatorname{Aut}(G)$-invariant partition
of $V(G)$.
\end{lemma}

\begin{proof}
Among all sets $Y\subseteq V(G)$ with
$2\le |Y|\le |V(G)|-2$, choose $S$ so that
\[
        b:=|\partial_GS|
\]
is minimum and, subject to this, the smaller side of the cut has minimum
order.  Replacing $S$ by its complement if necessary, put
$s:=|S|\le |V(G)|/2$. Then $s\ge 2$. The set $X$ is admissible, and hence
\begin{equation}\label{eq:restricted-cut-bound}
        b\le|\partial_GX|\le2d-3.
\end{equation}
Moreover, $s\ge3$: if $s=2$, simplicity gives
$b=2d-2e_G(S)\ge2d-2$, a contradiction to
\eqref{eq:restricted-cut-bound}.

For completeness, the edge-boundary function $f(Y):=|\partial_GY|$ is
submodular and posimodular:
\begin{align*}
        f(Y)+f(Z)&\ge f(Y\cap Z)+f(Y\cup Z),\\
        f(Y)+f(Z)&\ge f(Y\setminus Z)+f(Z\setminus Y).
\end{align*}
Both inequalities follow by checking the contribution of each edge; the
second also follows from the first and the symmetry
$f(Y)=f(V(G)\setminus Y)$.

We claim that $S$ is a block of imprimitivity.  Let
$T=\varphi(S)$ for some $\varphi\in\operatorname{Aut}(G)$, and suppose
that $S\cap T\ne\emptyset$ and $S\ne T$.  Since $|S|=|T|$, write
\[
        a:=|S\setminus T|=|T\setminus S|\ge1.
\]
If $a\ge2$, both differences are admissible sets.  Posimodularity of the
edge-boundary function gives
\[
        2b=|\partial_GS|+|\partial_GT|
        \ge |\partial_G(S\setminus T)|
          +|\partial_G(T\setminus S)|
        \ge2b.
\]
Thus $|\partial_G(S\setminus T)|=b$, a contradiction with the minimal choice of
$s$ because $|S\setminus T|<s$.  Hence $a=1$.  Now
$|S\cap T|=s-1\ge2$, while
\[
        |V(G)\setminus(S\cup T)|
        \ge s-1\ge2
\]
because $s\le |V(G)|/2$.  Thus $S\cap T$ and $S\cup T$ are admissible.
Submodularity gives
\[
        2b\ge |\partial_G(S\cap T)|+|\partial_G(S\cup T)|\ge2b,
\]
so $|\partial_G(S\cap T)|=b$, again contradicting the minimality of $s$.
Therefore every translate of $S$ is either equal to or disjoint from $S$.
Vertex-transitivity now shows that the translates partition $V(G)$, and
$S$ is a nontrivial block.

The setwise stabilizer $\operatorname{Aut}(G)_S$ acts transitively on $S$.
Indeed, if $x,y\in S$, choose an automorphism taking $x$ to $y$; its image
of $S$ meets $S$ and hence equals $S$.  It follows that every vertex of $S$
has the same number $q\ge1$ of neighbours outside $S$.  Consequently,
\begin{equation}\label{eq:block-external-degree}
        b=qs,
        \qquad
        d_{G[S]}(v)=d-q\quad(v\in S).
\end{equation}
The bound $b<2d-2$ and $s\ge3$ imply $q\le d-1$.  If $q\ge2$, then
simplicity and \eqref{eq:block-external-degree} give
\[
        s\ge d-q+1
\]
and therefore
\[
        b=qs\ge q(d-q+1)
        =2d-2+(q-2)(d-1-q)\ge2d-2,
\]
a contradiction.  Thus $q=1$.  It follows that
\[
        b=s,\qquad d_{G[S]}(v)=d-1\quad(v\in S).
\]
Simplicity yields $s\ge d$, while
\eqref{eq:restricted-cut-bound} yields $s\le2d-3$.

Finally, $G[S]$ is connected.  Indeed, if $G[S]$ were
disconnected, then it would have at least two components.  Since
$G[S]$ is $(d-1)$-regular, each of these components is itself
$(d-1)$-regular.  A simple $(d-1)$-regular graph has at least $d$
vertices, so $s=|S|\ge2d$, a contradiction with $s\le2d-3$.  Taking $B=S$ proves all the
assertions.
\end{proof}

We also record the required matching property in the host graph.

\begin{proposition}[Godsil and Royle
	{\cite[Theorem~3.5.1]{GodsilRoyle}}]\label{prop:host-parity-perfect}
Let $G$ be a connected vertex-transitive graph. If $|V(G)|$ is even,
then $G$ has a perfect matching. If $|V(G)|$ is odd, then $G$ is
factor-critical.
\end{proposition}
Indeed, in the odd-order case, vertex-transitivity moves the unique
unmatched vertex of a near-perfect matching to any prescribed vertex.

\subsection{Regular factors and counting tools}

\begin{lemma}[Ore's bipartite $f$-factor theorem
	{\cite{Ore1957}}]\label{lem:bipartite-factor}
	Let $J$ be a bipartite graph with parts $L,R$, where
	$|L|=|R|=n$, and let $r$ be a nonnegative integer.
	Then $J$ has a spanning $r$-regular subgraph if and only if
	\[
	e_J(X,R\setminus Y)\ge r(|X|-|Y|)
	\]
	for every $X\subseteq L$ and $Y\subseteq R$.
	It is enough to check pairs with $|X|>|Y|$.
\end{lemma}

This is the specialization $f\equiv r$ of Ore's bipartite
$f$-factor theorem. The condition is automatic when
$|X|\le |Y|$. The next lemma converts an even-regular graph into $2$-factors.

\begin{lemma}[Petersen's $2$-factorization theorem \cite{Petersen1891}]\label{lem:even-factorization}
	Every finite $2r$-regular graph admits a decomposition into
	$r$ pairwise edge-disjoint spanning $2$-factors.
\end{lemma}

We finish with two counting lemmas.  The first counts short collections of
near-minimum cuts, while the second counts vertex two-colourings with few
monochromatic edges.

\begin{lemma}\label{lem:collection-cuts}
Let $G$ be an $n$-vertex graph with edge-connectivity $\lambda$, and let
$h\ge1$ be an integer.  The number of ordered collections
$(X_1,\dots,X_r)$, where
$r\le h$, of nonempty proper vertex sets satisfying
$\sum_{i=1}^r |\partial_G X_i|\le h\lambda$ is at most $n^{6h}$.
\end{lemma}

\begin{proof}
For each $i$, choose the integer $b_i\ge 1$ such that
\[
        b_i\lambda\le |\partial_G X_i| < (b_i+1)\lambda.
\]
Since every nontrivial cut has size at least $\lambda$, the condition gives
$\sum_{i=1}^r b_i\le h$.  For a fixed sequence $(b_1,\dots,b_r)$,
Theorem~\ref{thm:karger} gives at most
$\prod_{i=1}^r 2n^{2(b_i+1)}
\le2^r n^{2\sum_i b_i+2r}\le2^h n^{4h}$ choices for the ordered collection.
The number of positive integer sequences of total at most $h$ is at most
$2^h$.  Hence the total is at most
\[
        4^h n^{4h}\le n^{6h},
\]
where we used $n\ge2$.
\end{proof}

\begin{lemma}\label{lem:near-bip-count}
Let $H$ be a loopless multigraph on at most $n$ vertices. Assume that
$H$ is $\lambda$-edge-connected, where $2\le\lambda\le n$.  For every
$b\ge1$,
the number of two-colourings
\[
        \chi:V(H)\to \{0,1\},
\]
up to global colour reversal, with at most $b\lambda$ monochromatic edges (two endpoints of the edge with the same colour) is
at most $n^{3(b+1)}$.
\end{lemma}

\begin{proof}
By Theorem~\ref{thm:treepacking}, fix
$r=\lfloor\lambda/2\rfloor$ edge-disjoint spanning trees
$T_1,\dots,T_r$.  Fix a colouring $\chi$ with at most $b\lambda$
monochromatic edges.  For a subgraph $T$ of $H$, let
$e_{\mathrm{mono}}(T,\chi)$ denote the number of its monochromatic edges.
Since the trees are edge-disjoint,
\[
        \sum_{j=1}^r e_{\mathrm{mono}}(T_j,\chi)
        \le b\lambda.
\]
Thus there is $j\in \{1,\ldots,r\}$ such that
\[
        e_{\mathrm{mono}}(T_j,\chi)
        \le \frac{b\lambda}{r}
        \le 3b.
\]
Fix this tree $T=T_j$ and let $R\subseteq E(T)$ be the set of monochromatic tree edges.  Then $|R|\le 3b$.  Once $R$ is known, the colouring of $T$ is determined up to global reversal: along edges of $T\setminus R$ colours must alternate, while along edges of $R$ colours must be equal.  Choosing the colour of one root vertex determines all other colours.

For this fixed tree, the number of possible sets $R$ is at most
\[
        \sum_{s=0}^{\lfloor 3b\rfloor}\binom{|V(H)|-1}{s}
        \le n^{3b+1}.
\]
For each admissible colouring, the averaging argument above gives at least
one index $j\in\{1,\ldots,r\}$ for which $T_j$ has at most $3b$
monochromatic edges.  The index $j$ is not fixed in advance.  We therefore
take the union over the $r$ possible choices of $j$, which multiplies the
upper bound by $r$.  A colouring may occur for more than one choice of $j$,
but this only causes harmless overcounting.  Since $r\le\lambda\le n$, the
resulting bound is at most
\[
        n^{3b+2}\le n^{3(b+1)}.
\]
\end{proof}

\begin{remark}
The condition $\lambda\le n$ is needed.  A multigraph on $n$ vertices may
have arbitrarily large edge multiplicity, so the estimate
$|E(H)|\le n^2$ is unavailable.  In our application,
$\lambda=d\le n-1$, where $d$ is the degree of the original simple graph.
\end{remark}

\section{Near-optimal 2-factor packings}\label{sec:two-factor}

\subsection{The bipartite double cover}

\begin{definition}\label{def:double-cover}
For a graph $G$, its \emph{bipartite double cover} $\widehat G$ has vertex
classes
\[
        V_L:=\{v_L:v\in V(G)\},
        \qquad
        V_R:=\{v_R:v\in V(G)\},
\]
and  $u_Lv_R\in E(\widehat G)$ whenever $uv\in E(G)$.  Thus every edge $uv$ of $G$
gives  two cover edges $u_Lv_R$ and $v_Lu_R$.  A walk $W=(v_0,v_1,\ldots,v_k)$ in $G$ has a unique lift in
$\widehat G$ starting at $v_{0,L}$, namely the walk obtained by
alternating the subscripts $L,R,L,R,\ldots$.
\end{definition}

The component structure of the bipartite double cover is important when the host is
bipartite.

\begin{lemma}\label{lem:double-cover}
Let $G$ be a connected $d$-regular graph.  Then $\widehat G$ is
$d$-regular and bipartite, and one of the following holds.
\begin{enumerate}[label=\textup{(\roman*)}]
\item If $G$ is nonbipartite, then $\widehat G$ is connected.
\item If $G$ is bipartite with bipartition $S\cup T$, then $\widehat G$
has exactly two  components, induced by
$S_L\cup T_R$ and $T_L\cup S_R$ respectively, and each component is isomorphic to $G$.
\end{enumerate}
If $G$ is also vertex-transitive, then every  component of
$\widehat G$ is vertex-transitive and has edge-connectivity $d$.
\end{lemma}

\begin{proof}
The $d$-regularity is immediate from Definition~\ref{def:double-cover}.
A walk in $G$ that starts at $u$ has a unique lift starting at $u_L$.  Its
lift ends in $V_L$ when the walk has even length and in $V_R$ when it has
odd length. Suppose that $G$ is nonbipartite.  It contains an odd cycle
$C$.  For any $u\in V(G)$, choose a path $P$ from $u$ to a vertex of $C$.
Following $P$, traversing $C$, and then following $P$ backwards gives an
odd closed walk based at $u$.  Now fix $u,v\in V(G)$ and a $u$--$v$ path
$Q$.  The path $Q$ and the concatenation of this odd closed walk with $Q$
are $u$--$v$ walks of opposite parities.  Their lifts from $u_L$ therefore
reach both $v_L$ and $v_R$.  Since $u$ and $v$ were arbitrary, every vertex
of $\widehat G$ is reachable from $u_L$, and $\widehat G$ is connected.

Suppose that $G$ is bipartite with parts $S,T$.  The parity of every walk
is determined by the parts containing its endpoints.  Therefore no edge of
$\widehat G$ joins $S_L\cup T_R$ to $T_L\cup S_R$, and the connectedness of
$G$ shows that each displayed set induces a  component.  The map
$s_L\mapsto s$ and $t_R\mapsto t$ is an isomorphism from the first
component to $G$, and the second component is handled in the same way.

Now assume that $G$ is vertex-transitive.  In the bipartite case, each
component is isomorphic to $G$.  In the nonbipartite case, every
$\varphi\in\operatorname{Aut}(G)$ lifts to an automorphism of $\widehat G$
by $v_L\mapsto\varphi(v)_L$ and $v_R\mapsto\varphi(v)_R$.  The map
$v_L\leftrightarrow v_R$ is also an automorphism.  These automorphisms act
transitively on $V(\widehat G)$.  The edge-connectivity assertion now
follows from Theorem~\ref{thm:mader}.
\end{proof}

\subsection{A spanning directed regular kernel}

For every edge $uv\in E(G)$, independently choose one of the three outcomes
\[
        \text{absent},\qquad u\to v,\qquad v\to u
\]
with probabilities $1-p,p/2,p/2$, respectively.  Let $\vec G_p$ be the
resulting random oriented graph.  Its underlying undirected graph has the
same distribution as $G_p$.  A spanning subdigraph $D$ of an oriented graph
is an \emph{$r$-regular directed kernel} if
\[
        d_D^+(v)=d_D^-(v)=r
        \qquad\text{for every }v\in V(D).
\]

\begin{theorem}\label{thm:directed-kernel}
Under the hypotheses and notation of Theorem~\ref{thm:packing}, the random
oriented graph $\vec G_p$ contains an $r$-regular directed kernel with
probability at least $1-n^{-A}$.
\end{theorem}

\begin{proof}
Let $\widehat G$ be the bipartite double cover of $G$, with sides $V_L,V_R$.
From $\vec G_p$, form a random bipartite subgraph $J\subseteq\widehat G$ by
putting
\[
        u_Lv_R\in E(J)
        \quad\Longleftrightarrow\quad
        u\to v\in E(\vec G_p).
\]
An $r$-factor of $J$ is equivalent to an $r$-regular directed kernel in
$\vec G_p$: the factor degree of $v_L$ is the outdegree of $v$, and the
factor degree of $v_R$ is the indegree of $v$.

By Lemma~\ref{lem:double-cover}, the graph $\widehat G$ has one or two
 components.  Fix one component
$\widehat G_0$ with bipartition $(L_0,R_0)$ and edge set $E_0$, and put
$J_0:=J[L_0,R_0]$.  Since
$\widehat G_0$ is regular and bipartite, $|L_0|=|R_0|$.  We bound the
probability that $J_0$ has no spanning $r$-factor.

By Lemma~\ref{lem:bipartite-factor}, a failure produces sets
$X\subseteq L_0$ and $Y\subseteq R_0$ such that
\begin{equation}\label{eq:flow-failure}
        e_{J_0}(X,R_0\setminus Y)<r\delta,
\end{equation}where
$\delta:=|X|-|Y|>0$.
Counting degrees on $X$ and $Y$ gives
\[
        d|X|=e_{\widehat G_0}(X,Y)
        +e_{\widehat G_0}(X,R_0\setminus Y),
        \qquad
        d|Y|=e_{\widehat G_0}(X,Y)
        +e_{\widehat G_0}(L_0\setminus X,Y),
\]
and hence
\begin{equation}\label{eq:cross-edge-difference}
        e_{\widehat G_0}(X,R_0\setminus Y)
        -e_{\widehat G_0}(L_0\setminus X,Y)=d\delta.
\end{equation}
Let $U:=X\cup Y\subseteq V(\widehat G_0)$.  The cut
$\partial_{\widehat G_0}U$ is the disjoint union of the edges from $X$ to
$R_0\setminus Y$ and the edges from $Y$ to $L_0\setminus X$.  Thus
\[
        |\partial_{\widehat G_0}U|
        =e_{\widehat G_0}(X,R_0\setminus Y)
        +e_{\widehat G_0}(L_0\setminus X,Y).
\]
Together with \eqref{eq:cross-edge-difference}, this yields
\[
        e_{\widehat G_0}(X,R_0\setminus Y)
        =\frac{|\partial_{\widehat G_0}U|+d\delta}{2}.
\]
In particular,
\begin{equation}\label{eq:cross-edge-lower-bounds}
        e_{\widehat G_0}(X,R_0\setminus Y)\ge d\delta,
        \qquad
        e_{\widehat G_0}(X,R_0\setminus Y)
        \ge\frac{|\partial_{\widehat G_0}U|}{2}.
\end{equation}

Put $Z:=e_{J_0}(X,R_0\setminus Y)$.  The two cover edges arising from one
host edge are dependent, so we group their contributions by host edges.
For each $e=uv\in E(G)$, define
\[
        M_e:=\{u_Lv_R,v_Lu_R\}
        \cap E_{\widehat G_0}(X,R_0\setminus Y),
        \qquad m_e:=|M_e|,
\]
and put $\xi_e:=|M_e\cap E(J_0)|$.  The orientation experiment selects at
most one of the two cover edges arising from $e$.  Hence
$\xi_e\in\{0,1\}$.  The variables $(\xi_e)_{e\in E(G)}$ are independent,
and
\[
\mathbb P(\xi_e=1)=
\begin{cases}
0,&m_e=0,\\
p/2,&m_e=1,\\
p,&m_e=2.
\end{cases}
\]
When $m_e=2$, either orientation of a retained host edge contributes
exactly one edge to $Z$.  Since
$\sum_e m_e=e_{\widehat G_0}(X,R_0\setminus Y)$, we have
\begin{equation}\label{eq:flow-mean}
        \mu:=\mathbb EZ
        =\frac p2e_{\widehat G_0}(X,R_0\setminus Y).
\end{equation}
By the definition of $r$ and the first inequality in
\eqref{eq:cross-edge-lower-bounds},
\[
        r\delta\le(1-\epsilon)\frac{pd}{2}\delta
        \le(1-\epsilon)\frac p2
        e_{\widehat G_0}(X,R_0\setminus Y)
        =(1-\epsilon)\mu.
\]
Lemma~\ref{lem:chernoff}, \eqref{eq:flow-mean}, and the second inequality in
\eqref{eq:cross-edge-lower-bounds} give
\[
\begin{aligned}
        \mathbb P(Z<r\delta)
        &\le\exp\left(-\frac{\epsilon^2\mu}{2}\right)\\
        &\le\exp\left(-\frac{\epsilon^2p}{8}
        |\partial_{\widehat G_0}U|\right).
\end{aligned}
\]

It remains to sum over the possible witnesses.  The set $U$ determines the
pair $(X,Y)$ uniquely, because $X=U\cap L_0$ and $Y=U\cap R_0$.
Moreover, $\delta>0$ implies that $U$ is nonempty and proper.  By
Lemma~\ref{lem:double-cover}, the edge-connectivity of $\widehat G_0$ is
$d$.  For an integer $j\ge1$, consider the sets $U$ satisfying
\[
        jd\le|\partial_{\widehat G_0}U|<(j+1)d.
\]
Theorem~\ref{thm:karger} gives fewer than
$2|V(\widehat G_0)|^{2(j+1)}\le2(2n)^{2(j+1)}$ such sets.  Hence the
probability that \eqref{eq:flow-failure} holds for some $X,Y$ in this
component is at most
\[
        \sum_{j\ge1}2(2n)^{2(j+1)}
        \exp\left(-\frac{\epsilon^2pd}{8}j\right).
\]
There are at most two components in $\widehat G$.  By \eqref{eq:packing-condition}, the
total failure probability is at most
\[
\begin{aligned}
        4\sum_{j\ge1}
        \exp\bigl(2(j+1)\log(2n)-8(A+6)j\log(2n)\bigr)
        &=4(2n)^2\sum_{j\ge1}(2n)^{-(8A+46)j}\\
        &\le8(2n)^{-(8A+44)}\\
        &\le n^{-A}.
\end{aligned}
\]
For the penultimate inequality we used $n\ge3$, which gives
$(2n)^{-(8A+46)}<1/2$.  Thus every component of $J$ has a spanning
$r$-factor simultaneously with probability at least $1-n^{-A}$.  Their union is an
$r$-factor of $J$, and hence $\vec G_p$ contains an $r$-regular directed
kernel.
\end{proof}

\subsection{The packing theorem and its consequences}

\begin{proof}[Proof of Theorem~\ref{thm:packing}]
Let $\mathbb P$ denote probability in the original percolation space of
$G_p$, and let $\widetilde{\mathbb P}$ denote probability in the enlarged
space in which $G_p$ is supplied with the auxiliary orientation used to
form $\vec G_p$.  Let $\mathcal K$ be the event that $\vec G_p$ contains
an $r$-regular directed kernel.  By
Theorem~\ref{thm:directed-kernel},
\[
\widetilde{\mathbb P}(\mathcal K)\ge 1-n^{-A}.
\]

On $\mathcal K$, choose an $r$-regular directed kernel
$D\subseteq\vec G_p$.  Forget the directions of the arcs in $D$ and call
the resulting undirected graph $H_D$.  Since $\vec G_p$ is oriented, no
host edge is used twice in opposite directions.  Thus $H_D$ is a simple
spanning subgraph of $G_p$, and
\[
d_{H_D}(v)=d_D^+(v)+d_D^-(v)=2r
\qquad\text{for every }v\in V(G).
\]
Lemma~\ref{lem:even-factorization} decomposes $H_D$ into $r$ pairwise
edge-disjoint spanning $2$-factors.

Let $\mathcal E$ be the event that $G_p$ contains $r$ pairwise
edge-disjoint spanning $2$-factors. Viewing $\mathcal E$ also as an event in the enlarged probability space,
it depends only on the underlying undirected graph.  Moreover,
$\mathcal K\subseteq\mathcal E$.  Since the underlying undirected graph
of $\vec G_p$ has the same distribution as $G_p$, we obtain
\[
\mathbb P(\mathcal E)
=\widetilde{\mathbb P}(\mathcal E)
\ge \widetilde{\mathbb P}(\mathcal K)
\ge 1-n^{-A}.
\]
\end{proof}

\begin{proof}[Proof of Corollary~\ref{cor:packing-asymptotic}]
For the lower bound, put
\[
        \epsilon_n:=\left(\frac{\log n}{p_nd_n}\right)^{1/4}.
\]
Then $\epsilon_n\to0$ and
\[
        \frac{\epsilon_n^2p_nd_n}{\log n}
        =\sqrt{\frac{p_nd_n}{\log n}}\longrightarrow\infty.
\]
For any fixed $A>0$, by Theorem~\ref{thm:packing}, for all sufficiently
large $n$, we have
\[
        \nu_2((G_n)_{p_n})
        \ge(1-o(1))\frac{p_nd_n}{2}
\]
with probability at least $1-n^{-A}$.

For the upper bound, each spanning $2$-factor has exactly $n$ edges, so
\[
        \nu_2((G_n)_{p_n})\le\frac{e((G_n)_{p_n})}{n}.
\]
The random variable $e((G_n)_{p_n})$ is binomial with mean
$np_nd_n/2$.  Its variance is at most its mean, which tends to infinity.
Chebyshev's inequality gives
\[
        e((G_n)_{p_n})=(1+o(1))\frac{np_nd_n}{2}
\]
with high probability.  Combining the two estimates proves the result.
\end{proof}

The relative-error form also gives a useful additive estimate.

\begin{corollary}\label{cor:additive-packing}
Let $G,n,d,p,A$ be as in Theorem~\ref{thm:packing}, and suppose that
$pd>64(A+6)\log(2n)$.  Then, with probability at least $1-n^{-A}$,
\[
        \nu_2(G_p)\ge
        \left\lfloor\frac{pd}{2}
        -4\sqrt{(A+6)pd\log(2n)}\right\rfloor.
\]
\end{corollary}

\begin{proof}
Choose
\[
        \epsilon
        :=\sqrt{\frac{64(A+6)\log(2n)}{pd}}.
\]
The assumption $pd>64(A+6)\log(2n)$ ensures that
$\epsilon\in(0,1)$, and
\[
        \epsilon^2pd=64(A+6)\log(2n).
\]
Therefore, Theorem~\ref{thm:packing} gives, with probability at least
$1-n^{-A}$,
\[
\begin{aligned}
        \nu_2(G_p)
        &\ge
        \left\lfloor(1-\epsilon)\frac{pd}{2}\right\rfloor\\
        &=
        \left\lfloor
        \frac{pd}{2}
        -4\sqrt{(A+6)pd\log(2n)}
        \right\rfloor,
\end{aligned}
\]
as required.
\end{proof}

\section{Parity obstructions and random parity-perfectness}
\label{sec:parity-obstructions}

\subsection{Parity obstructions}

We now prepare to prove Theorem~\ref{thm:main-poly} by describing the
minimal edge sets whose deletion destroys parity-perfectness.  Let $G$ be
parity-perfect, and put
$\tau\equiv|V(G)|\pmod2$.  A set $F\subseteq E(G)$ is a
\emph{parity-preclusion set} if $G-F$ is not parity-perfect.  It is
\emph{minimal} if no proper subset has this property.

Let $F$ be a minimal parity-preclusion set and put $J:=G-F$. We have the following results about the properties of $J$.
\begin{lemma}\label{lem:def-two}
 We have
$\Delta_\tau(J)=2$.  Consequently, every parity barrier $U$ of $J$ satisfies
\begin{equation}\label{eq:barrier-count}
        \odd(J-U)=|U|+2-\tau.
\end{equation}
\end{lemma}

\begin{proof}
The graph $J$ is not parity-perfect, and every value in
\eqref{eq:parity-deficiency} is even.  Hence
$\Delta_\tau(J)\ge2$.  Adding one edge can merge at most two
components after deleting any fixed set $U$, so it can decrease the parity
deficiency by at most two.  If $\Delta_\tau(J)\ge4$, then, for every
$e\in F$,
\[
        \Delta_\tau(J+e)\ge2.
\]
This contradicts the minimality of $F$.  Thus
$\Delta_\tau(J)=2$, and \eqref{eq:barrier-count} follows from the
definition of a parity barrier.
\end{proof}

\begin{lemma}\label{lem:max-barrier-no-even}
Let $U$ be a parity barrier of $J$ that is maximal under inclusion.  That is,
no parity barrier properly contains $U$.  Then every component of $J-U$ is
odd.
\end{lemma}

\begin{proof}
Suppose that $R$ is an even component of $J-U$.  Choose $x\in V(R)$ and put
$U':=U\cup\{x\}$.  Since $R-x$ has odd order, it has at least one odd
component.  Hence
\[
        \odd(J-U')\ge\odd(J-U)+1.
\]
It follows that
\[
        \odd(J-U')-|U'|+\tau
        \ge \odd(J-U)-|U|+\tau=2.
\]
By the maximality of the parity deficiency, equality holds.  Thus $U'$ is also
a parity barrier, a contradiction  with the maximality of $U$ under inclusion.
\end{proof}

\begin{lemma}\label{lem:F-between-odd}
Let $U$ be a parity barrier of $J$ that is maximal under inclusion, and let
$C_1,\ldots,C_r$ be the components of $J-U$. Then every edge of $F$ joins two
distinct components $C_i,C_j$.  Conversely, every edge of $G$ joining two
distinct such components belongs to $F$.
\end{lemma}

\begin{proof}
By Lemmas~\ref{lem:max-barrier-no-even} and~\ref{lem:def-two}, all $C_i$
are odd and $r=|U|+2-\tau$.

Take $e\in F$.  By the minimality, $J+e$ is parity-perfect.  If $e$ is incident with $U$, then it disappears after deleting $U$. So
$
        (J+e)-U=J-U.
$
Consequently,
\[
        \odd((J+e)-U)-|U|+\tau
        =\odd(J-U)-|U|+\tau
        =2.
\]
Thus $\Delta_\tau(J+e)\ge2$, a contradiction with the parity-perfectness of
$J+e$.  If both endpoints of $e$ lie in one component
$C_i$, then adding $e$ does not change the components of $(J+e)-U$, and
again $\Delta_\tau(J+e)\ge2$.  Both alternatives are impossible.
Hence $e$ joins two distinct components.

Conversely, an edge of $G$ between distinct components $C_i,C_j$ cannot
belong to $J$, because they are distinct components of $J-U$.  It therefore
belongs to $F$.
\end{proof}

\begin{lemma}
Assume that $G$ is $d$-regular and $d$-edge-connected.  In the setting
above, put $t:=|F|$.  Then
\begin{align}
        e_G(U)&\le t-\frac{(2-\tau)d}{2},
        \label{eq:barrier-induced-bound}\\
        \sum_{i=1}^r\bigl(|\partial_GV(C_i)|-d\bigr)
        &\le 2t-(2-\tau)d.
        \label{eq:boundary-excess}
\end{align}
\end{lemma}

\begin{proof}
For $1\le i\le r$, put
\[
        u_i:=e_G(V(C_i),U),
        \qquad
        f_i:=e_G(V(C_i),V(G)\setminus(U\cup V(C_i))).
\]
By Lemma~\ref{lem:F-between-odd}, we have
 $\sum_{i=1}^r f_i=2t$.  Since $G$ is $d$-regular,
\[
        \sum_{i=1}^r u_i=d|U|-2e_G(U).
\]
Note that each $V(C_i)$ is a nonempty proper set.  Hence the edge-connectivity and
\eqref{eq:barrier-count} give
\[
\begin{aligned}
        d|U|-2e_G(U)+2t
        &=\sum_{i=1}^r|\partial_GV(C_i)|\\
        &\ge rd=d(|U|+2-\tau).
\end{aligned}
\]
This is \eqref{eq:barrier-induced-bound}.  Subtracting $rd$ from the exact
expression for the sum of the boundaries gives
\[
        \sum_{i=1}^r\bigl(|\partial_GV(C_i)|-d\bigr)
        =2t-(2-\tau)d-2e_G(U),
\]
which implies \eqref{eq:boundary-excess}.
\end{proof}

\subsection{Random parity-perfectness}\label{sec:random}

\begin{definition}
A connected $d$-regular graph $G$ on $n$ vertices is irreducible if every set $X\subseteq V(G)$ with
$2\le |X|\le n-2$ and $G[X]$ being connected satisfies
$|\partial_GX|\ge2d-2$.
\end{definition}

\begin{proposition}\label{prop:irr-count}
Let $G$ be an irreducible connected simple vertex-transitive graph of order
$n$ with degree $d\ge4$.  For every integer
$a\ge1$, the number of minimal parity-preclusion sets $F$ satisfying
\[
        \frac{ad}{2}\le |F|<\frac{(a+1)d}{2}
\]
is at most $n^{21(a+1)}$.
\end{proposition}

\begin{proof} Recall $\tau\equiv n\pmod2$.
Fix such a set $F$, put $J:=G-F$ and $t:=|F|$, and choose a parity barrier
$U$ that is maximal under inclusion, meaning that no parity barrier of $J$
properly contains $U$.  Let $C_1,\ldots,C_r$ be the
components of $J-U$.  By Lemmas~\ref{lem:max-barrier-no-even}
and~\ref{lem:def-two}, all $C_i$ are odd and
\[
        r=|U|+2-\tau.
\]
By \eqref{eq:boundary-excess},
\begin{equation}\label{eq:excess-irreducible}
        \sum_{i=1}^r\bigl(|\partial_GV(C_i)|-d\bigr)
        <(a-1+\tau)d\le ad.
\end{equation}

Fix a labelling of $V(G)$, and order collections of disjoint sets by their
least labelled vertices.  Put
\[
        \mathcal L:=\{i:|V(C_i)|>1\}.
\]
If $2\le|V(C_i)|\le n-2$, by the  irreducibility, we have
\[
        |\partial_GV(C_i)|-d\ge d-2.
\]
There is at most one component of order $n-1$.  Since $d\ge4$,
\[
        |\mathcal L|\le\frac{ad}{d-2}+1\le2a+1.
\]

By \eqref{eq:excess-irreducible},
\[
        \sum_{i\in\mathcal L}|\partial_GV(C_i)|
        \le d|\mathcal L|+ad\le(3a+1)d.
\]
Lemma~\ref{lem:collection-cuts}, with $h=3a+1$, gives at most
\[
        n^{6(3a+1)}=n^{18a+6}
\]
choices for the ordered collection of non-singleton components.  Fix one
such collection.

Contract each fixed non-singleton component to one vertex and delete loops.
Let $H$ be the resulting loopless multigraph.  Every cut of $H$ lifts to a
cut of $G$, so $H$ is at least $d$-edge-connected.  The barrier data gives a
two-colouring of $V(H)$: the vertices representing the components $C_i$
have colour $1$, and the vertices of $U$ have colour $0$.

The colour-$1$ monochromatic edges are exactly the $t$ edges of $F$ between
distinct components.  The colour-$0$ monochromatic edges are the edges
induced by $U$.  By \eqref{eq:barrier-induced-bound}, the total number of
monochromatic edges is at most
\[
        t+e_G(U)
        \le2t-\frac{(2-\tau)d}{2}
        <\left(a+\frac{\tau}{2}\right)d
        \le(a+1)d.
\]
Lemma~\ref{lem:near-bip-count}, with $\lambda=d$ and $b=a+1$, gives at most
$n^{3a+6}$ colourings up to reversal, meaning that $\chi$ and $1-\chi$ are
identified.  If a non-singleton component was contracted,
that vertex is prescribed to have colour $1$, so exactly one of these two
representatives has the required colours.  If every component is a
singleton, we retain both representatives, which changes the count by a
factor of at most two.

Once the non-singleton components and the chosen representative of the
colouring are fixed, the
singleton components and $U$ are fixed.  Lemma~\ref{lem:F-between-odd} then
determines
\[
        F=\bigcup_{i<j}E_G(V(C_i),V(C_j))
\]
uniquely.  The extra factor of two is at most $n$, so the total number is
at most
\[
        n^{18a+6}n^{3a+6}n=n^{21a+13}
        \le n^{21(a+1)}.
\]
\end{proof}

\begin{proposition}\label{prop:irr-random}
Let $\beta>0$, and  $G$ be an irreducible connected simple
vertex-transitive graph of order $n$ and degree $d\ge4$.  Let
$N\ge n$.  If
\[
        q^d\le N^{-(2\beta+90)},
\]
then
\[
        \mathbb P(G_p\text{ is not parity-perfect})\le N^{-\beta}.
\]
\end{proposition}

\begin{proof}
Let $D:=E(G)\setminus E(G_p)$.  By
Proposition~\ref{prop:host-parity-perfect},  $G$ is parity-perfect.
If $G-D$ is not parity-perfect, then $D$ contains a minimal
parity-preclusion set $F$.

For a fixed $F$ in the $a$-th layer of Proposition~\ref{prop:irr-count},
\[
        \mathbb P(F\subseteq D)=q^{|F|}
        \le(q^d)^{a/2}
        \le N^{-(\beta+45)a}.
\]
Therefore
\[
\begin{aligned}
        \mathbb P(G_p\text{ is not parity-perfect})
        &\le\sum_{a\ge1}N^{21(a+1)-(\beta+45)a}\\
        &=\sum_{a\ge1}N^{21-(\beta+24)a}
          \le N^{-\beta}.
\end{aligned}
\]
Indeed, the first term is $N^{-\beta-3}$ and the ratio is
$N^{-(\beta+24)}$; since $N\ge n\ge d+1\ge5$, the final inequality follows.
\end{proof}

Assume that $G$ is not irreducible and $d\ge4$.  By
Lemma~\ref{thm:block},  $\operatorname{Aut}(G)$ has a nontrivial block system $\mathcal B$
with common block size, say $m$. Call an edge of $G$ \emph{internal} if its endpoints lie in the same block of $\mathcal B$, and \emph{external} otherwise. For every $B\in\mathcal B$,  $G[B]$
is connected and $(d-1)$-regular, and
\begin{equation}\label{eq:block-size}
        d\le m\le2d-3.
\end{equation}
Thus every vertex has exactly one neighbour outside its block.  Consequently,
the external edges form a perfect matching of $V(G)$, and in particular
$n$ is even.  If $m$ is odd, then the number $n/m$ of blocks is therefore
even as well.

Let $Q^*$ be the loopless multigraph obtained by contracting every block to a single vertex and keeping all external edges as parallel edges.  Let $Q$ be the underlying simple graph of $Q^*$.

\begin{lemma}\label{lem:constant-mult}
The simple quotient $Q$ is connected and vertex-transitive.  Moreover, there is an integer $\ell\ge 1$ such that every edge of $Q$ corresponds to exactly $\ell$ parallel edges of $Q^*$.  If $k$ is the degree of $Q$, then $\ell k=m$.
\end{lemma}

\begin{proof}
The connectivity of $Q$ follows from the connectivity of $G$.  The action of
$\operatorname{Aut}(G)$ on the blocks is transitive, so $Q$ is
vertex-transitive.

Fix $B\in\mathcal B$.  Its setwise stabilizer
$\operatorname{Aut}(G)_B$ is transitive on $B$: if $x,y\in B$, choose
$\varphi\in\operatorname{Aut}(G)$ with $\varphi(x)=y$.  The blocks $B$ and
$\varphi(B)$ meet, so they are equal.  Since each vertex of $B$ has a unique
external neighbour, $\operatorname{Aut}(G)_B$ is transitive on the external
darts leaving $B$.

Hence every neighbour block of $B$ receives the same number, say $\ell$, of
external edges from $B$.  Block transitivity makes $\ell$ independent of
$B$.  The $m$ external edges leaving $B$ split into $k$ classes of size
$\ell$, and therefore $\ell k=m$.
\end{proof}

The graph induced by each block has order $m$ and degree
$d-1>m/2$.  We next prove the random completion statement needed inside
these blocks.  It uses the obstruction count above and requires no separate
quantitative random-Dirac theorem.

\begin{proposition}\label{prop:dense-completion}
Let $\beta>0$, and  $R$ be a connected simple vertex-transitive graph of
order $M\ge2$ and degree $r\ge M/2$.  Let $N\ge M$.  If
\[
        q^M\le N^{-(5\beta+240)},
\]
then
\[
        \mathbb P(R_p\text{ is not parity-perfect})\le N^{-\beta}.
\]
\end{proposition}

\begin{proof}
Put $C_D:=5\beta+240$.  Suppose first that $r<4$.  Then $M\le2r\le6$.
If $M$ is even, every minimal parity-preclusion set $F$ has
$|F|\ge r\ge M/2$ by \eqref{eq:boundary-excess}.  There are at most
$2^{|E(R)|}\le2^{15}$ possible sets.  Hence
\[
        \mathbb P(R_p\text{ has no perfect matching})
        \le 2^{15}q^{M/2}
        \le N^{15-C_D/2}
        \le N^{-\beta}.
\]
If $M$ is odd, then the regularity makes $r$ even.  The assumptions force
$r=2$ and $M=3$, so $R=K_3$.  This graph remains factor-critical exactly
when all three edges survive.  Therefore
\[
        \mathbb P(R_p\text{ is not factor-critical})
        \le3q\le N^{2-C_D/3}\le N^{-\beta}.
\]

Assume that $r\ge4$.  If $R$ is irreducible, then
\[
        q^r=(q^M)^{r/M}\le N^{-C_D/2}
        \le N^{-(2\beta+90)},
\]
and the result holds by Proposition~\ref{prop:irr-random}.

Now suppose that $R$  is not irreducible.  Lemma~\ref{thm:block} gives a proper
block of size $s\ge r\ge M/2$.  Since there are at least two blocks,
$s\le M/2$.  Therefore
\[
        s=r=M/2.
\]
So there are exactly two blocks, and each induced block graph is $K_s$.

First suppose that $s\ge5$.  The graph $K_s$ is irreducible: if
$2\le|X|\le s-2$, then
\[
        |\partial_{K_s}X|=|X|(s-|X|)
        \ge2s-4=2(s-1)-2.
\]
Moreover,
\[
        q^{s-1}=(q^M)^{(s-1)/(2s)}
        \le N^{-2C_D/5}
        \le N^{-(2\beta+94)}.
\]
Apply Proposition~\ref{prop:irr-random}, with error exponent $\beta+2$, to
both copies of $K_s$.  Thus both random copies are parity-perfect except
with probability at most $2N^{-(\beta+2)}$.

If $s$ is even, the two internal perfect matchings form a perfect matching
of $R_p$.  Suppose that $s$ is odd.  The external edges form a perfect
matching between the two blocks, and the probability that all of them are
deleted is
\[
        q^s=(q^M)^{1/2}\le N^{-C_D/2}.
\]
If an external edge $xy$ survives, use factor-criticality to choose perfect
matchings inside the two blocks after deleting $x$ and $y$, and then add
$xy$.  The total failure probability is at most
\[
        2N^{-(\beta+2)}+N^{-C_D/2}\le N^{-\beta}.
\]

It remains to consider $s=4$.  For each copy of $K_4$, every minimal set
whose deletion destroys all perfect matchings has size at least three, and
there are at most $2^6$ such sets.  Consequently, the probability that one
of the two random copies has no perfect matching is at most
\[
        2^7q^3\le N^{7-3C_D/8}\le N^{-\beta}.
\]
The union of two internal perfect matchings then completes a perfect
matching of $R_p$.
\end{proof}

Expose the external edges of $G_p$.  Declare an edge of $Q$ open if at
least one of its $\ell$ original edges survives.  The open quotient is
$Q_\rho$, where
\begin{equation}\label{eq:quotient-rho}
        \rho:=1-q^\ell.
\end{equation}
Distinct quotient edges correspond to disjoint sets of original edges, so
the open events are independent.  By Lemma~\ref{lem:constant-mult},
\begin{equation}\label{eq:quotient-transform}
        (1-\rho)^k=(q^\ell)^k=q^{\ell k}=q^m\le q^d.
\end{equation}

\begin{proposition}\label{prop:reducible-lifting}
Let $\beta>0$.  Whenever $N\ge n$ and
\[
        q^d\le N^{-(5\beta+245)},
\]
the following hold.
\begin{enumerate}[label=\textup{(\roman*)}]
\item If $m$ is even, then
\[
        \mathbb P(G_p\text{ is not parity-perfect})\le N^{-\beta}.
\]
\item If $m$ is odd, then
\[
        \mathbb P(G_p\text{ is not parity-perfect})
        \le\mathbb P(Q_\rho\text{ has no perfect matching})+N^{-\beta}.
\]
\end{enumerate}
\end{proposition}

\begin{proof}
For every block $B_0\in\mathcal B$,  $G[B_0]$ is connected and
vertex-transitive under the setwise stabilizer of $B_0$.  It has order $m$
and degree $d-1$.  By \eqref{eq:block-size},
\[
        d-1\ge\frac{m+1}{2}>\frac m2.
\]
Also $q^m\le q^d$, and
\[
        5(\beta+1)+240=5\beta+245.
\]
Apply Proposition~\ref{prop:dense-completion} to every block with error
exponent $\beta+1$.  Since there are at most $n\le N$ blocks, the following
event $\mathcal E$ has probability at least $1-N^{-\beta}$: every random
block graph has a perfect matching when $m$ is even and is factor-critical
when $m$ is odd.

Suppose that $m$ is even. On $\mathcal E$, choose a perfect matching
$M_B$ of $G_p[B]$ for every $B\in\mathcal B$. Since the blocks partition
$V(G)$, the union
$
\bigcup_{B\in\mathcal B} M_B
$
is a perfect matching of $G_p$. This proves~(i).

Now suppose that $m$ is odd.  As observed above, the external edges of $G$
form a perfect matching.  Hence $n$ is even and the quotient order $n/m$ is
even.  Work on the intersection of $\mathcal E$ with the event that
$Q_\rho$ has a perfect matching.  For every edge of such a matching, choose
one surviving original external edge.  This matches one vertex in each
block.  Factor-criticality supplies a perfect matching of the remaining
vertices in every block, and the union is a perfect matching of $G_p$.
A union bound proves (ii).  Notice that an odd-order branch cannot occur.
\end{proof}

\subsection{Proof of Theorem~\ref{thm:main-poly}}

\begin{proof}[Proof of Theorem~\ref{thm:main-poly}]
Fix $A>0$ and put
\[
        \beta=A+1,
        \qquad
        C:=5A+250=5\beta+245,
        \qquad
        L(r)=1+\lfloor\log_3 r\rfloor\quad(r\ge2).
\]
We prove by strong induction on $n$ the stronger estimate
\begin{equation}\label{eq:induction-claim}
        \mathbb P(G_p\text{ is not parity-perfect})
        \le L(n)N^{-\beta}.
\end{equation}
Assume that \eqref{eq:induction-claim} holds for all smaller orders.

We first dispose of $d<4$.  If $d=1$, then $G=K_2$ and its unique edge
survives except with probability
\[
        q\le N^{-C}\le N^{-\beta}.
\]
If $d=2$, then $G=C_n$.  For even $n$, its two alternating perfect
matchings give
\[
\begin{aligned}
        \mathbb P(G_p\text{ has no perfect matching})
        =(1-p^{n/2})^2
        \le (nq/2)^2
         \le N^2q^2
         \le N^{2-C}
         \le N^{-\beta}.
\end{aligned}
\]
For odd $n$, the random cycle is factor-critical only if every edge
survives, and hence
\[
        \mathbb P(G_p\text{ is not factor-critical})
        =1-p^n\le nq
        \le N(q^2)^{1/2}
        \le N^{1-C/2}
        \le N^{-\beta}.
\]
Finally, if $d=3$, then $n$ is even.  Fix a perfect matching of $G$ using
Proposition~\ref{prop:host-parity-perfect}.  If all its edges survive, then
$G_p$ has a perfect matching. So
\[
        \mathbb P(G_p\text{ has no perfect matching})
        \le\frac n2q
        \le N(q^3)^{1/3}
        \le N^{1-C/3}
        \le N^{-\beta}.
\]
This proves \eqref{eq:induction-claim} when $d<4$.  Now we assume $d\ge4$.

If $G$ is irreducible, then
$C=5\beta+245\ge2\beta+90$, so
Proposition~\ref{prop:irr-random}, used with error exponent $\beta$, gives
failure probability at most $N^{-\beta}$.

Suppose that $G$  is not irreducible, and let $m$ be the block size given by
Lemma~\ref{thm:block}.  If $m$ is even,
Proposition~\ref{prop:reducible-lifting}(i) gives failure probability at most
$N^{-\beta}$.

It remains to treat odd $m$.  Since the block system is nontrivial,
$m\ge3$.  The simple quotient $Q$ is connected and vertex-transitive.  Its
order
\[
        n_Q=|V(Q)|=\frac nm
\]
is even and satisfies $n_Q\le n/3<n$.  Let $k$ be the
degree of $Q$, and let $\rho$ be defined by
\eqref{eq:quotient-rho}.  By \eqref{eq:quotient-transform},
\[
        (1-\rho)^k\le q^d\le N^{-C}.
\]
The induction hypothesis, with the same ambient parameter $N$, gives
\[
        \mathbb P(Q_\rho\text{ has no perfect matching})
        \le L(n_Q)N^{-\beta}.
\]
Proposition~\ref{prop:reducible-lifting}(ii) now yields
\[
        \mathbb P(G_p\text{ is not parity-perfect})
        \le \bigl(L(n_Q)+1\bigr)N^{-\beta}.
\]
Because $n_Q\le n/3$,
\[
        L(n_Q)+1
        \le L(n).
\]
This closes the induction.

Since $L(n)\le N$ for $N\ge n\ge2$, \eqref{eq:induction-claim} and
$\beta=A+1$
give
\[
        \mathbb P(G_p\text{ is not parity-perfect})
        \le N\,N^{-(A+1)}
        \le N^{-A}.
\]
\end{proof}

\section{Concluding remarks}\label{sec:scope}

The two results show that the isolated-vertex scale already forces the
matching and local degree structures contained in a Hamilton cycle.  The
remaining difficulty is global: one must connect the cycles of a
$2$-factor without losing the probability scale.

\subsection{Algorithmic and structural consequences}

The packing in Theorem~\ref{thm:packing} can be found in randomized
polynomial time.  Given $G_p$, orient each retained edge independently and
uniformly.  Construct the random bipartite graph $J$ on the double cover.
For each component of the double cover, use the standard max-flow algorithm for
the bipartite $f$-factor problem in
Lemma~\ref{lem:bipartite-factor}.  On the event in
Theorem~\ref{thm:directed-kernel}, the algorithm returns an $r$-factor.
After forgetting directions, orient Euler tours and decompose the resulting
regular bipartite graph into perfect matchings; this is the constructive
proof of Lemma~\ref{lem:even-factorization}.  The procedure produces the
full packing with the same success probability.

The packing argument uses vertex-transitivity only to obtain maximal edge
connectivity in the components of the double cover.  Thus it gives the
following extension.

\begin{proposition}\label{prop:double-cover-extension}
Let $G$ be a connected simple $n$-vertex $d$-regular graph such that every
 component of its bipartite double cover $\widehat G$ has
edge-connectivity $d$.  Then the conclusion of Theorem~\ref{thm:packing}
holds for $G$ with the same constants.
\end{proposition}

\begin{proof}
In the proof of Theorem~\ref{thm:directed-kernel}, replace the final
assertion of Lemma~\ref{lem:double-cover} by the assumed edge-connectivity
of the cover components.  Every other step is unchanged.
\end{proof}

In contrast, the parity-perfectness theorem uses vertex-transitivity both
in the obstruction count and in the recursive block decomposition.

\subsection{The isolated-vertex obstruction is not uniformly sharp}

The polynomial condition in Theorem~\ref{thm:main-poly} does not say that
isolated vertices are the only obstruction in every vertex-transitive graph.

\begin{proposition}\label{prop:cycle-exponent-obstruction}
For even cycles,
\begin{equation}\label{eq:cycle-probability}
        \mathbb P((C_{2m})_p\text{ has a perfect matching})
        =2p^m-p^{2m}.
\end{equation}
For odd cycles,
\begin{equation}\label{eq:odd-cycle-probability}
        \mathbb P((C_{2m+1})_p\text{ is factor-critical})
        =p^{2m+1}.
\end{equation}
In particular, if $q=1-p=c/m$ for a fixed $c>0$, then the probability in
\eqref{eq:cycle-probability} tends to
\[
        2e^{-c}-e^{-2c}<1,
\]
whereas the probability that $G_p$ has an isolated vertex tends to zero.
\end{proposition}

\begin{proof}
The cycle $C_{2m}$ has exactly two perfect matchings, namely its two
alternating edge sets.  Each survives with probability $p^m$, and both
survive exactly when every edge survives, an event of probability $p^{2m}$.
Inclusion--exclusion gives \eqref{eq:cycle-probability}.

An odd cycle is factor-critical.  Deleting any one of its edges leaves an
odd path, which is not factor-critical; deleting further edges cannot
restore the property.  Hence the percolated odd cycle is factor-critical
exactly when every edge survives, proving
\eqref{eq:odd-cycle-probability}.

If $q=c/m$, then $p^m\to e^{-c}$.  A fixed vertex is isolated with
probability $q^2$, so the expected number of isolated vertices is
\[
        2mq^2=\frac{2c^2}{m}\longrightarrow0.
\]
Markov's inequality completes the proof.
\end{proof}

Thus a sharp threshold theorem must either assume that the degree grows or
include further low-degree matching-preclusion structures.

The cycle calculation also locates the possible optimal exponent.  With
$C_*(A)$ as defined after Theorem~\ref{thm:main-poly}, we have
\begin{equation}\label{eq:exponent-window}
        2A+2\le C_*(A)\le5A+250.
\end{equation}
Indeed, the upper bound is Theorem~\ref{thm:main-poly}.  For the lower
bound, take an odd $n$ and $G=C_n$, put $N=n$, and set
$q=n^{-C/2}$.  Then $q^d=n^{-C}$, whereas
\[
        \mathbb P(G_p\text{ is not factor-critical})
        =1-(1-q)^n.
\]
If $C<2A+2$, this probability is asymptotically larger than $n^{-A}$
(and it does not even tend to zero when $C\le2$).  Thus no smaller uniform
exponent is possible.  If one restricts to even-order hosts, the even-cycle
formula similarly gives the lower bound $C\ge A+2$.

\subsection{A deterministic cycle-lifting lemma}

Theorem~\ref{thm:packing} controls the degrees in each spanning factor, but
not its number of cycles.  A Hamilton cycle has one component, whereas the
max-flow inequalities impose only the local degree conditions of a cycle
cover.  The block decomposition also shows what is needed to lift a
Hamilton cycle.

\begin{lemma}\label{lem:hamilton-lifting}
Let $G$ have a block system $\mathcal B$ in which every vertex has exactly
one neighbour outside its block, and let $H\subseteq G$ be a spanning
subgraph.  Suppose first that the simple quotient on $\mathcal B$ has at
least three vertices.  Assume that
\begin{enumerate}[label=\textup{(\roman*)}]
\item The open quotient of $H$ contains a Hamilton cycle $Z$.
\item For every edge of $Z$, one corresponding edge of $H$ is chosen.
\item For every block $B$, the two chosen external edges incident with $B$
have distinct endpoints $x_B,y_B\in B$, and $H[B]$ has a Hamilton path
from $x_B$ to $y_B$.
\end{enumerate}
Then $H$ has a Hamilton cycle.

If the quotient has two vertices, the same conclusion holds when two
distinct external edges are chosen and each block has a Hamilton path
between the endpoints of those two edges.
\end{lemma}

\begin{proof}
Replace each quotient vertex $B$ on $Z$ by a Hamilton path in $H[B]$ joining
the endpoints of the two incident chosen external edges.  The internal paths
are vertex-disjoint and cover $V(H)$.  The chosen external edges join them in
the cyclic order of $Z$.  Their union is therefore a spanning cycle.  The
two-block case is the same, with the two external edges joining the two
internal Hamilton paths.
\end{proof}

Because every vertex has only one external edge, distinct external edges
incident with a block have distinct endpoints.  The remaining random input
would be simultaneous Hamilton-connectedness inside the blocks and a
Hamiltonian quotient theorem stable under repeated passage to
vertex-transitive or Schreier quotients.  The obstruction count used here
does not provide an analogue of Tutte barriers for Hamiltonicity.

\subsection{Further questions}

\begin{problem}
Find natural degree conditions under which
\[
        nq^d\longrightarrow0
\]
implies parity-perfectness with high probability.
Proposition~\ref{prop:cycle-exponent-obstruction} shows that this does not
hold uniformly for bounded degree.
\end{problem}

Lemma~\ref{lem:hamilton-lifting} isolates one subproblem in the reducible
case.  The irreducible case and robust Hamilton-connectedness remain
separate obstacles.  Theorems~\ref{thm:main-poly} and~\ref{thm:packing}
prove two factor-theoretic consequences of
Conjecture~\ref{conj:random-hamilton}, but do not control the cycle-merging
step required for Hamiltonicity.

\begin{problem}
Close the gap in \eqref{eq:exponent-window}.  In particular, determine
whether the lower bound $C_*(A)=2A+2$ is sharp, and classify the minimal
parity-preclusion sets that determine the leading failure probability.
\end{problem}

\section*{Acknowledgement}
The authors acknowledge the use of AI tools during the exploratory stage of this project. All mathematical arguments and proofs in the final manuscript were checked and written by the authors.

\end{document}